\documentclass{amsart}

\usepackage[utf8]{inputenc}
\usepackage{enumerate}
\usepackage{amsfonts,amsmath}
\usepackage{amssymb}
\usepackage{amsthm,amscd}
\usepackage{setspace} 
\usepackage{hyperref}
\usepackage{cite}
\usepackage{enumerate}

\newtheorem{thm}{Theorem}[section]
 \newtheorem{cor}[thm]{Corollary}
 \newtheorem{lem}[thm]{Lemma}
 
 \theoremstyle{definition}
 
 \theoremstyle{remark}
 \newtheorem{rem}[thm]{Remark}
 \newtheorem{ex}[thm]{Example}

\def\b{{\beta}}
\def\g{{\gamma}}
\def\G{{\Gamma}}

\def\R{{\mathbb R}}
\def\deg{{\rm deg}}
\def\ID2{$(\text{ID}_2)$}
\def\sgn{{\rm sgn}}
\def\lc{{\rm l.c.}}

\begin{document}

\title[]{Idempotent factorization of matrices over a Pr\"ufer domain of rational functions}

\author{Laura Cossu}

\address{Laura Cossu\\ Institut für Mathematik und Wissenschaftliches Rechnen\\
 Karl-Franzens-Universität Graz\\
Heinrichstraße 36, 8010 Graz\\ Austria}

\email{laura.cossu@uni-graz.at, laura.cossu.87@gmail.com}

\thanks{The author is a member of the Gruppo Nazionale per le Strutture Algebriche, Geometriche e le loro Applicazioni (GNSAGA) of the Istituto Nazionale di Alta Matematica (INdAM)}

 

\begin{abstract}
We consider the smallest subring $D$ of $\R(X)$ containing every element of the form $1/(1+x^2)$, with $x\in \R(X)$. $D$ is a Pr\"ufer domain called the {\it minimal Dress ring of $\R(X)$}. In this paper, addressing a general open problem for Pr\"ufer non B\'ezout domains, we investigate whether $2\times 2$ singular matrices over $D$ can be decomposed as products of idempotent matrices. We show some conditions that guarantee the idempotent factorization in $M_2(D)$.
\end{abstract}

\maketitle
\section{Introduction}
In 1965 Andreas Dress \cite{D} introduced a family of Pr\"ufer domains constructed as subrings $D_K$ of a field $K$ containing every element of the form $1/(1+x^2)$, for $x\in K$. Given a field $K$ not containing square roots of $-1$, the subring of $K$ generated by $\{(1+x^2)^{-1}\,:\, x\in K\}$ is said to be the {\it minimal Pr\"ufer-Dress ring} (or simply the {\it minimal Dress ring}) of $K$. We refer to \cite{D} and \cite{CZ_Dress} for more details on these domains. In the paper \cite{CZ_Dress}, the authors investigated minimal Dress rings of special classes of fields: Henselian fields, ordered fields and formally real fields (e.g., $\R(\mathcal{A})$, with $\mathcal{A}$ a set of indeterminates). They focused in particular on the minimal Dress ring $D$ of the field of real rational functions $\R(X)$, characterizing its elements \cite[Prop. 2.1]{CZ_Dress} and ideals \cite[Prop. 2.4]{CZ_Dress} and proving that $D$ is a Dedekind domain (i.e., a Noetherian Pr\"ufer domain) that is not a principal ideal domain \cite[Th. 2.3]{CZ_Dress}. They also identified a family of $2\times 2$ singular matrices over $D$ that can be written as a product of idempotent factors \cite[Th. 3.3]{CZ_Dress}. The study of the factorization of singular square matrices over rings as product of idempotent matrices has raised a remarkable interest both in the commutative and non-commutative setting since the middle of the 1960's (see \cite{SalZan,survey}). We say that an integral domain $R$ satisfies the property \ID2 if every $2\times 2$ singular matrix over $R$ is a product of idempotent factors.
A natural and well motivated conjecture, proposed by Salce and Zanardo in \cite{SalZan} and then investigated in \cite{CZ_Prufer} and \cite{CZ_integers}, asserts that every domain $R$ satisfying \ID2 must be a B\'ezout domain, namely, every finitely generated ideal of $R$ must be principal. Note that the reverse implication is false: not every B\'ezout domain verifies \ID2 (see \cite{Cohn, CZZ}). In \cite{CZ_Prufer} it is proved that if $R$ satisfies \ID2, then every finitely generated ideal of $R$ is invertible and so $R$ is a Pr\"ufer domain. Therefore, it is not restrictive to study \ID2 within this class of domains and, in view of the above conjecture, we expect that for every Pr\"ufer non-B\'ezout domain $R$ there exists at least one singular matrix in $M_2(R)$ that cannot be written as a product of idempotent factors.

In this paper we develop the investigation started in \cite{CZ_Dress} on idempotent factorizations of $2\times 2$ matrices over the minimal Dress ring $D$ of $\R(X)$. In Section 2 we fix the notation and recall some preliminary results and definitions. In Section 3 we focus on the factorizations in $M_2(D)$ and, in Theorems \ref{positive}, \ref{odd_1} and \ref{even_1}, we identify several conditions on a couple of elements $p,q\in D$ under which the matrix $\begin{pmatrix}
p & q\\
0 & 0
\end{pmatrix}$ factors into idempotents. In this way we supplement the results in \cite{CZ_Dress} by providing further families of $2 \times 2$ matrices over $D$ that admit idempotent factorizations. Moreover, in Example \ref{Example} we exhibit a singular matrix in $M_2(D)$ for which the failure of the above conditions prevents an ``easy'' decomposition into idempotent factors. However the general problem whether $D$ satisfies \ID2 remains open.

\section{Preliminaries and notation}

Let $R$ be an integral domain. We will use the standard notations $R^{\times}$ to denote its multiplicative group of units and $M_n(R)$ to denote the $R$-algebra of $n\times n$ matrices over $R$.

A square matrix $\mathbf{T}$ over $R$ is said to be idempotent if $\mathbf{T}^2=\mathbf{T}$. A direct computation shows that a singular nonzero matrix $\begin{pmatrix}
a & b\\
c & d
\end{pmatrix}$ over an arbitrary integral domain is idempotent if and only if $d = 1 - a$. For a singular matrix $\mathbf{S}\in M_n(R)$, the property of being a product of idempotent factors is preserved by similarity. This immediately leads to the following lemma.

\begin{lem}[Lemma 3.1 of \cite{CZ_Dress}]\label{scambio}
Let $R$ be an integral domain, $p, q \in R$. The matrix
$\begin{pmatrix}
p & q\\
0 & 0
\end{pmatrix}$
is a product of idempotent matrices if and only if such is
$\begin{pmatrix}
q & p\\
0 & 0
\end{pmatrix}.$
\end{lem}

The next result will also be useful in the following.

\begin{lem}\label{construction} 
Let $p$ and $q$ be nonzero elements of an integral domain $R$, and $\textbf{M}=\begin{pmatrix}
p & q\\
0 & 0
\end{pmatrix}\in M_2(R)$. If $\textbf{M}=\textbf{S}\cdot\textbf{T}$, with $\textbf{S}=\begin{pmatrix}
p' & q'\\
z & t
\end{pmatrix}$ a singular matrix and $\textbf{T}=\begin{pmatrix}
a & b\\
c & 1-a
\end{pmatrix}$ an idempotent matrix over $R$, then $\textbf{S}$ has the form $\textbf{S}=\begin{pmatrix}
p' & q'\\
0 & 0
\end{pmatrix}$.
\end{lem}
We omit the proof, since it is essentially contained in that of Lemma 3.1 in \cite{CZZ}.

Finally, we recall below two immediate factorizations in $M_2(R)$:
\begin{equation}\label{r0}
\begin{pmatrix}
p & 0\\
0 & 0
\end{pmatrix}=\begin{pmatrix}
1 & -1\\
0 & 0
\end{pmatrix}\begin{pmatrix}
1 & 0\\
1-p & 0
\end{pmatrix} \ ; \ \begin{pmatrix}
0 & q\\
0 & 0
\end{pmatrix}=\begin{pmatrix}
1 & 0\\
0 & 0
\end{pmatrix}\begin{pmatrix}
0 & q\\
0 & 1
\end{pmatrix}.
\end{equation}

\medskip

From now on $D$ will denote the minimal Dress rings of the field of rational functions $\R(X)$. 

Following the notation in \cite{CZ_Dress}, let $\G$ be the set of the polynomials in $\R[X]$ that have no roots in $\R$. Then $\G = \{\alpha\prod_i \g_i \}$, where the $\g_i$ are monic degree-two polynomials irreducible over $\R[X]$ and $0 \ne \alpha$ is a real number. Set $\G^+ = \{f \in \R[X] : f(r) > 0 , \forall \, r \in \R \}$ and, correspondingly, $\G^- =  \{- f  : f \in \G^+ \}$. By Proposition 2.1 in \cite{CZ_Dress}, 
\[D = \{f/\g : f \in \R[X], \g\in \G , \deg \, f \le \deg \, \g \},\]
and
\[D^{\times}=\{ \g_1/\g_2: \g_1,\g_2\in \G, \deg \, \g_1 = \deg \, \g_2 \}.\]
As recalled in the introduction, we know from Theorem 2.3 and Proposition 2.4 of \cite{CZ_Dress} that $D$ is a Dedekind domain which is not a principal ideal domain. As an example, the ideal generated by $1/\g$ and $X/\g$, with $\g\in\G \setminus \mathbb R$, is not principal. 

It is worth remarking that a non-constant polynomial of $\mathbb R[X]$ {\it never} lies in $D$. 

Given a polynomial $f \in \R[X]$ we will denote as $\lc(f)$ its leading coefficient. In the following, given an element $p=f/\g\in D$ we will always assume that $\gamma$ is a product of monic irreducible polynomials of degree $2$. We will then refer to $\lc(f)$ as the leading coefficient of $p$.

\section{Idempotent factorizations in $M_2(D)$}

In this section we investigate property \ID2 over $D$. We find sufficient conditions on the entries of a  singular matrix over $D$ to get a factorization into idempotents. 

We start recalling two results of \cite{CZ_Dress} that we will need in our discussion.

\begin{lem}[Lemma 3.2 of \cite{CZ_Dress}]\label{grado uguale}
Let $x,y$ be non-zero polynomials in $\R[X]$ with $\deg\,x=\deg\,y$. 

\begin{enumerate}[(a)]

\item If $y(u)>0$ (or $y(u)<0$) for every $u$ root of $x$, then there exists $\beta \in \G$ such that $\delta=x^2+ y \beta\in \G^+$, $\deg\, x - 1 \le \deg \b \le \deg \, x = \deg\,\delta/2$.

\item If $x(z)>0$ (or $x(z)<0$) for every $z$ root of $y$, then there exists $\eta \in \G$ such  that $\delta= x\eta+y^2 \in \G^+$ and $\deg \, y - 1 \le \deg\,\eta \le \deg\,y=\deg \delta/2$.
\end{enumerate}
\end{lem}

\begin{thm}[Th. 3.3 of \cite{CZ_Dress}]\label{teorema 3.3}
Let $p$, $q$ be elements of $D$. Then the matrix $\begin{pmatrix}
p & q\\
0 & 0
\end{pmatrix}$ is a product of idempotent matrices if one of the following holds:
\begin{enumerate}[(i)]
\item $\deg\,p \ge \deg\,q$ and $q(u)>0$ (or $q(u)< 0$) for every $u$ root of $p$

\item $\deg\,q \ge \deg\,p$ and $p(z)>0$ (or $p(z)< 0$) for every $z$ root of $q$.
\end{enumerate}
\end{thm}

Two polynomials $x,y\in\R[X]$ are said to be \textit{weakly comaximal} if $\gcd(x,y)\in\G$, i.e., if $x$ and $y$ have no common roots.
Let $p$ and $q$ be two elements of $D$. Then we can always write $p=x/\g$ and $q=y/\g$, with $\g\in \G^+$ and $x,y\in \R[X]$. We say that $p$ and $q$ are \textit{weakly comaximal} if so are $x$ and $y$.

\begin{thm}\label{positive}
Let $p$ and $q$ be weakly comaximal elements of $D$. If either $p\geq 0$  or $q\geq 0$, then the matrix $\begin{pmatrix}
p & q\\
0 & 0
\end{pmatrix}$ is a product of idempotent matrices.
\end{thm}
\begin{proof}
By Lemma \ref{scambio}, $\begin{pmatrix} 
p & q\\
0 & 0
\end{pmatrix}$ is a product of idempotent matrices if and only if such is $\begin{pmatrix}
q & p\\
0 & 0
\end{pmatrix}$, therefore we can safely assume that $p\geq 0$.

Let us first consider the case $\deg \, p \geq \deg \, q$.
We can assume without loss of generality that $\deg \, p =\deg \, q$. In fact, since $\begin{pmatrix}
p & q\\
0 & 0
\end{pmatrix}$ is similar to $\begin{pmatrix}
p & p+q\\
0 & 0
\end{pmatrix}$ and $p$ and $p+q$ are still weakly comaximal, if $\deg \, p > \deg\, q$, it not restrictive to replace $q$ with $p + q$. Thus, let $\deg \, p = \deg \, q$ and set $p = x/\g$ and $q = y/\g$, with $x, y \in \R [X]$, $\g \in \G^+$. 

As a further reduction, we may assume that $\deg \,p = \deg \,q=0$. In fact, being $p\geq 0$, every root of $p$ has even multiplicity and $\deg\, x$ is even. Moreover, if $\deg\,x<\deg\,\g$, taking any $\tau \in \G^+$ such that $\deg \, x = \deg \,\tau$,
$$
\begin{pmatrix}
x/\g & y/\g\\
0 & 0
\end{pmatrix} =
\begin{pmatrix}
\tau/\g & 0\\
0 & 0
\end{pmatrix}
\begin{pmatrix}
x/\tau & y/\tau\\
0 & 0
\end{pmatrix}
$$
is a factorization in $M_2(D)$ and, by (\ref{r0}), the matrix on the left of the above equality is a product of idempotents if such is the second factor of the product on the right. 

Since for every $z$ root of $y$, $x(z)$ is always $>0$, we have got in the position to apply Lemma \ref{grado uguale} (ii) to $x$ and $y$. Therefore, there exists $\eta \in \G$ such that
$
\delta = x \eta + y^2 \in \G^+
$
where $\deg \, \eta = \deg \, x$ and $\deg \, \delta = 2 \deg \, \eta$. 

Then, since
$
\deg \, x=\deg \, y=\deg \,\g=\deg\,\eta  
$, $\delta/\g \eta \in D^\times$ and $x\eta/\delta$, $y\eta/\delta$, $x y/\delta$, $y^2/\delta \in D$.
Moreover, the relation $1 - x\eta/\delta = y^2/\delta$ implies that $\mathbf{T}=\begin{pmatrix}
x\eta/\delta & y\eta/\delta\\
x y/\delta & y^2/\delta
\end{pmatrix}$ is an idempotent matrix over $D$. Therefore
\[\begin{pmatrix}
x/\g & y/\g \\
0 & 0
\end{pmatrix}=\begin{pmatrix}
\delta/\g\eta & 0 \\
0 & 0
\end{pmatrix}\mathbf{T},
\] and using (\ref{r0}) we conclude that $\begin{pmatrix}
p & q\\
0 & 0
\end{pmatrix}$ is a product of idempotent matrices over $D$.

On the other hand, if $\deg \,q> \deg\, p$, it suffices to apply Theorem \ref{teorema 3.3} (ii).
\end{proof}

\begin{rem}
The matrix $\begin{pmatrix}
p & q\\
0 & 0
\end{pmatrix}\in M_2(D)$ is a product of idempotent matrices even if $p$ and $q$ are two comaximal elements of $D$ such that either $p\leq 0$ or $q\leq 0$. The proof is basically the same as that of Theorem \ref{positive}.
\end{rem}

In what follows the symbol $f^{(j)}$ denotes the $j$-th derivative of the polynomial  $f \in \mathbb R[X]$.

\begin{lem}\label{baselemma}
Let $x,y\in \R[X]$ and $\varepsilon\in \R^+$ be such that:
\begin{itemize}
\item $y$ has $0$ as unique root with multiplicity $k$;
\item $x(0)\neq 0$;
\item $y^{(i)}>0$ in $(0, \varepsilon]$ for $0\leq i\leq k-1$ ;
\item $y^{(k)}$ does not change sign in $(0, \varepsilon]$;
\item $x^{(j)}$ is either zero or does not change sign in $(0, \varepsilon]$ for $0\leq j\leq k$.
\end{itemize}
Then, there exists a real number $r_0>0$ such that, for every $r\in ( 0,  r_0 ]$, $rx+y$ has at most one root in $(0,\varepsilon]$, and exactly one root if $x<0$ in $( 0,\varepsilon]$.
\end{lem}
\begin{proof}
Note that $y^{(k)}(0) \ne 0$, hence $y^{(k)}$ has nonzero max and min in $[0, \varepsilon]$. Since, by assumption, $x^{(k)}$ is either zero or does not change sign in $(0, \varepsilon]$, an easy direct check shows that, for all possible signs, there exists $r_0 > 0$ such that, for every $r\in ( 0, r_0 ]$, $r x^{(k)} + y^{(k)}$ is either strictly positive or negative in $[0, \varepsilon]$. It follows that $\forall r\in ( 0, r_0 ]$, $r x^{(k-1)} + y^{(k-1)}$ is either increasing or decreasing in the interval and hence it has at most one root.

Now, let us consider the $(k-1)$-th derivative of $rx+y$.

We distinguish three cases.
\begin{enumerate}[(i)]
\item If $x^{(k-1)}\geq 0$ in $(0,\varepsilon]$, then $rx^{(k-1)}+y^{(k-1)}> 0$ in this neighborhood for every $r\in (\ 0,\ r_0\ ]$. Therefore, being increasing, $rx^{(k-2)}+y^{(k-2)}$ has at most a unique root in the interval.
\item If $x^{(k-1)}<0$ in $( 0, \varepsilon]$ and $y^{(k-1)}(0)\neq 0$, then, by possibly choosing a smaller $r_0$, $rx^{(k-1)}+y^{(k-1)}$  is either strictly positive or negative in $[0, \varepsilon]$ for every $r\in (0, r_0]$, and again we get that $rx^{(k-2)}+y^{(k-2)}$ has at most a unique root.
\item If $x^{(k-1)} <0$ in $( 0, \varepsilon]$  and $y^{(k-1)}(0)= 0$, by possibly choosing a smaller $r_0$, $rx^{(k-1)}(\varepsilon)+y^{(k-1)}(\varepsilon)>0$ for every $r\in ( 0, r_0]$. Since  $rx^{(k-1)}(0)+y^{(k-1)}(0)\leq 0$, we have two possibilities: or $rx^{(k-1)}+y^{(k-1)}>0$ in $(0,\varepsilon]$ for every $r\in ( 0, r_0]$, or there exists $x_{k-1}^r\in ( 0, \varepsilon)$ such that $rx^{(k-1)}(x_{k-1}^r)+y^{(k-1)}(x_{k-1}^r)=0$ and this zero is unique since $r x^{(k-1)} + y^{(k-1)}$ has at most one root. As a consequence, in the first case $rx^{(k-2)}+y^{(k-2)}$ is strictly increasing and admits at most one root, in the second case it decreases on $(0, x_{k-1}^r)$ and increases on $(x_{k-1}^r, \varepsilon]$.
\end{enumerate}

Let us now distinguish three more cases for the $(k-2)$-th derivative of $rx+y$.

\begin{enumerate}[(i)]
\item If $x^{(k-2)}\geq 0$ in $(0,\varepsilon]$, then $rx^{(k-2)}+y^{(k-2)}> 0$ in this neighborhood for every $r\in ( 0, r_0 ]$, therefore $rx^{(k-3)}+y^{(k-3)}$ is increasing and it has at most a unique root in the interval.
\item If $x^{(k-2)}<0$ in $( 0, \varepsilon]$ and $y^{(k-2)}(0)\neq 0$, then, by possibly choosing a smaller $r_0$, $rx^{(k-2)}+y^{(k-2)}$  is either strictly positive or negative in $[0, \varepsilon]$ for every $r\in (0, r_0]$, and again we get that $rx^{(k-3)}+y^{(k-3)}$ has at most a unique root.
\item If $x^{(k-2)} <0$ in $( 0, \varepsilon]$  and $y^{(k-2)}(0)= 0$, by possibly choosing a smaller $r_0$ $rx^{(k-2)}(\varepsilon)+y^{(k-2)}(\varepsilon)>0$ for every $r\in(0,r_0]$. Since $rx^{(k-2)}(0)+y^{(k-2)}(0)\leq 0$ we have two possibilities: $rx^{(k-2)}+y^{(k-2)}>0$ in $(0,\varepsilon]$ for every $r\in(0,r_0]$ or, for every $r\in(0,r_0]$, there exists $x_{k-2}^r\in ( 0, \varepsilon)$ zero of $rx^{(k-2)}+y^{(k-2)}$. By the previous step $rx^{(k-2)}+y^{(k-2)}$ is either increasing or has a unique critical point on $(0,\varepsilon]$. In both this cases we cannot have other roots besides $x_{k-2}^r$. As a consequence $rx^{(k-3)}+y^{(k-3)}$ is either strictly increasing or it decreases on $(0, x_{k-2}^r)$ and increases on $(x_{k-2}^r, \varepsilon]$.
\end{enumerate}

Iterating the procedure, after $k$ steps we obtain that there exists a real number $r_0>0$ such that, for every $r\in ( 0,  r_0 ]$, $rx+y$ has at most a unique root in $(0,\varepsilon]$ and exactly one root if $x<0$ in $( 0,\varepsilon]$.
\end{proof}

\begin{rem}
In the hypothesis of the above Lemma, if $k=1$, the proof becomes much easier. If $x>0$ in $(0,\varepsilon]$, $rx+y>0$ for every positive $r\in \R$. Let us assume henceforth that $x < 0$ in $(0,\varepsilon]$. There always exists a suitable $r_0>0$ such that $rx(\varepsilon)+y(\varepsilon)>0$ for every $r\in(0,r_0]$. Since $y(0)=0$ and $y(\varepsilon)>0$, it must be $y'>0$ on $(0,\varepsilon]$. If in the same interval $x'\geq 0$ then $rx'+y'>0$ and since $rx(0)+y(0)<0$, $rx+y$ has a unique root in $(0,\varepsilon]$ for every $r\in(0,r_0]$. If $x'<0$ in $(0,\varepsilon]$, by possibly choosing a smaller $r_0$, $rx'+y'$ is still strictly positive in $(0,\varepsilon]$ for every $r\in(0,r_0]$ and we conclude as before.
\end{rem}

\begin{lem}\label{1root_dispari}
Let $x,y$ be polynomials in $\R[X]$ without common roots, such that $\deg\,x$ and  $\deg\,y$ are odd, $\deg\,x>\deg\,y$, $y$ has a unique root and there exist $x_1,x_2\in \R$ roots of $x$ such that $y(x_1)y(x_2)<0$. Then, there exists a suitable $r\in\R$, such that $rx+y$ has a unique root.
\end{lem}

\begin{proof}
It is not restrictive to assume $\lim_{X\to \pm\infty} xy=+\infty$. If the leading coefficients $\lc(x)$ and $\lc(y)$ are discordant the proof can be accordingly adapted by replacing $r$ with $-r$.

Let $\lim_{X\to \pm\infty} x, y=\pm \infty$. The case $\lim_{X\to \pm\infty} x, y=\mp \infty$ is analogous.
Up to a suitable translation we can assume $y(0)=0$, $x_1<0$ and $x_2>0$. Let $k$ be the (odd) multiplicity of $0$ as a root of $y$. Let $I_0=(-\varepsilon, \varepsilon)$ be a sufficiently small neighborhood of $0$ such that $x < 0$ and $y$ is strictly increasing in $I_0$. The case $x>0$ in $I_0$ can be treated similarly in the interval $[-\varepsilon,0)$. By possibly choosing a smaller $\varepsilon$, we may assume that in the interval $(0, \varepsilon]$ $y^{(i)}>0$ for $1\leq i\leq k-1$, $y^{(k)}$ does not change sign and $x^{(j)}$ is either zero or does not change sign for $1\leq j\leq k$. By Lemma \ref{baselemma} there exists a real number $r_0>0$ such that, for any $r \in (0, r_0]$, $r x+y$ has a unique root on $[0,\varepsilon]$. Let us observe that, in $[-\varepsilon,\ 0\ ]$, $rx+y<0$ for every positive $r$.

Now take $M \in \R^+ $ such that $xy>0$ for all $X$ such that $|X|>M$. Clearly, for every $r>0$, $rx+y$ has no roots for $|X|>M$. We consider the closed intervals $I_1=[\ -M,\ -\varepsilon\ ]$ and $I_2=[ \ \varepsilon,\  M\ ]$. Let $M_x=\max_{I_1\cup I_2} |x| \,>0 $ and  $m_y=\min_{I_1\cup I_2} |y|$. Since $y$ has no roots other than $0$, clearly $m_y>0$. By choosing $0<r<m_y/M_x$ the polynomial $rx+y$ has no zeroes in $I_1\cup I_2$. 

We conclude by choosing any $0<r<\min\{r_0, m_y/M_x\}$.
\end{proof}

\begin{thm}\label{odd_1}
Let $p$ and $q$ be elements of $D$. If $\deg\,p, \deg\,q$ are odd and either $p$ or $q$ has a unique root, then the matrix $\begin{pmatrix}
p & q\\
0 & 0
\end{pmatrix}$ is a product of idempotent matrices.
\end{thm}

\begin{proof}
By Lemma \ref{scambio}, we can safely assume that $q$ has a unique root. 

We first consider the case $p$ and $q$ weakly comaximal.

If $\deg\, q \geq \deg\,p$, since $q$ has a unique root and $p$ and $q$ have no common factors, the hypothesis of Th. \ref{teorema 3.3} (ii) are satisfied.

If $\deg\, p > \deg\,q$ we distinguish two cases. If $q$ does not change sign on the roots of $p$, we are done by Theorem \ref{teorema 3.3} (i). Otherwise,
by Lemma \ref{1root_dispari}, it is always possible to find a suitable $r\in \R$ such that $rp+q$ has a unique root. Therefore, by Theorem \ref{teorema 3.3} (ii), the matrix $\begin{pmatrix}
p & rp+q\\
0 & 0
\end{pmatrix}$, similar to $\begin{pmatrix}
p & q\\
0 & 0
\end{pmatrix}$, is a product of idempotent matrices.

\medskip

Now consider the case of $p$ and $q$ not weakly comaximal. If $p=x/\g$ and $q=y/\g$, with $\g\in \G^+$ and $x,y\in \R[X]$, $x$ and $y$ have a common root. Since $q$ has odd degree and a unique root $z\in \R$, we have $x=(X-z)^h\bar{x}$ and $y=(X-z)^k \bar{y}$ with $k,h$ positive integers, $k$ odd, $\bar{y}\in\G$, $\bar{x}\in\R[X]$ and $\gcd(X-z,\bar{x})=1$.
Let us choose any $\delta \in \G^+$ such that either $\deg \delta=\min\{k,h\}$ or $\deg \delta=\min\{k,h\}+1$, in accordance with the parity of $\min\{k,h\}$. Since $\max\{\deg p, \deg q\}<0$, 
$\begin{pmatrix}
p & q\\
0 & 0
\end{pmatrix}=\begin{pmatrix}
(X-z)^{\min\{k,h\}}/\delta & 0\\
0 & 0
\end{pmatrix}\begin{pmatrix}
(X-z)^{h-\min\{k,h\}}\bar{x}\delta/\g & (X-z)^{k-\min\{k,h\}}\bar{y}\delta/\g\\
0 & 0
\end{pmatrix}$
is a factorization in $M_2(D)$ and, by (\ref{r0}), $\begin{pmatrix} p & q\\
0 & 0
\end{pmatrix}$ is a product of idempotent matrices if such is $\mathbf{S}=\begin{pmatrix}
(X-z)^{h-\min\{k,h\}}\bar{x}\delta/\g & (X-z)^{k-\min\{k,h\}}\bar{y}\delta/\g\\
0 & 0
\end{pmatrix}$. Let us remark that the elements of the first row of $\mathbf{S}$, $(X-z)^{h-\min\{k,h\}}\bar{x}\delta/\g$ and $(X-z)^{k-\min\{k,h\}}\bar{y}\delta/\g$, are now weakly comaximal.

If $h\geq k$, then
$
\mathbf{S}=\begin{pmatrix}
(X-z)^{h-k}\bar{x}\delta/\g &\bar{y}\delta/\g\\
0 & 0
\end{pmatrix}
$
and, since $\bar{y}\delta\in \G$, we conclude by applying Theorem \ref{positive}.

If $k>h$, we get
$
\mathbf{S}=\begin{pmatrix}
\bar{x}\delta/\g &(X-z)^{k-h}\bar{y}\delta/\g\\
0 & 0
\end{pmatrix}.
$
 If $h$ is even, then $\deg\,\bar{x}$ is odd, since $\deg\,x=h+\deg\,\bar{x}$ is odd. Moreover also $k-h$ is odd. It follows that $\bar{x}\delta/\g$ and $(X-z)^{k-h}\bar{y}\delta/\g$ are two weakly comaximal element of $D$ with odd degree and, being $\bar{y}$ and $\delta$ elements of $\G$, $(X-z)^{k-h}\bar{y}\delta/\g$ has a unique root $z\in\R$. Therefore, from the first part of the proof, we conclude that $\mathbf{S}$ is a product of idempotent matrices. If $h$ is odd, being $k-h$ even, $(X-z)^{k-h}\bar{y}\delta$ is always $\geq 0$ or $\leq 0$ and we conclude by applying Theorem \ref{positive}.

All possible cases have been examined.
\end{proof}

\begin{lem}\label{1root_pari}
Let $x,y$ be polynomials in $\R[X]$ without common roots, such that $\deg\,x$ is even, $\deg\,y$ is odd, $\deg\,x>\deg\,y$, $y$ has a unique root $y_1$ and there exist $x_1,x_2\in \R$ roots of $x$ such that $y(x_1)y(x_2)<0$. Then, there exists a suitable $r\in\R$, such that $rx+y$ has exactly two distinct roots $z_1,z_2\in \R$. Moreover, if the sign of $x(y_1)$ and that of the leading coefficient of $x$ are the same (resp. opposite), then $x(z_1)x(z_2)>0$ (resp. $x(z_1)x(z_2)<0$).
\end{lem}

\begin{proof}
We assume that $\lim_{X\to \pm\infty} x=+\infty$ and $\lim_{X\to \pm\infty} y=\pm\infty$. If the leading coefficients $\lc(x)$ and $\lc(y)$ have opposite signs or they are both negative, the proof can be easily adapted.

Up to a suitable translation we can assume $y(0)=0$, $x_1<0$ and $x_2>0$. Let $k$ be the (odd) multiplicity of $0$ as root of $y$. Let $I_0=(-\varepsilon, \varepsilon)$ be a sufficiently small neighborhood of $0$ such that $x<0$ in $I_0$ and, in $(0,\varepsilon]$, $y^{(i)}>0$ for $1 \leq i\leq k-1$, $y^{(k)}$ does not change sign and $x^{(j)}$ is either zero or does not change sign  for $1\leq j\leq k$. The case $x>0$ in $I_0$ can be similarly treated in the neighborhood $[-\varepsilon, 0)$. Under the above assumptions, by Lemma \ref{baselemma} there exists a real number $r_0>0$ such that, for any assigned $r \in (0, r_0]$, $r x+y$ has a unique root in $[\ 0,\ \varepsilon\ ]$. Let us observe that in $[-\varepsilon,\ 0\ ]$ $rx+y<0$ for every positive $r$.

Take $M \in \R^+ $ such that $x,y>0$ in the interval $[M, + \infty)$. Clearly, for every $r>0$, $rx+y$ has no roots in $(0, M)$. 

Now we choose $N \in \R^+ $ such that $x>0$, $y<0$ for $X\leq -N$ and $x'>0$ and $y'<0$ in the interval $(-\infty,-N)$. Under these assumptions, there exists a real number $r_1>0$ such that, for any assigned $r \in (0, r_1]$, $r x'+y'<0$ and $(rx+y)(-N)<0$. Therefore, $rx+y$ has a unique root in $(-\infty,-N ]$ for every $r\in (0,r_1]$. 

Let us consider the closed intervals $I_1=[\ -N,\ -\varepsilon\ ]$ and $I_2=[ \ \varepsilon,\  M\ ]$. Let $M_x=\max_{I_1\cup I_2} |x| \,>0 $ and  $m_y=\min_{I_1\cup I_2} |y| \,>0$ ($y$ has no roots other than $0$). By choosing $0<r<m_y/M_x$ the polynomial $rx+y$ has no zeroes in $I_1\cup I_2$. 

We can conclude by choosing any $0<r<\min\{r_0, r_1, m_y/M_x\}$.

The last statement of the theorem follows immediately by construction.
\end{proof}

\begin{thm}\label{even_1}
Let $p$ and $q$ be two weakly comaximal elements of $D$. If $\deg\,p$ is even, $\deg\,q$ is odd and either $p$ has a unique root or $q$ has a unique root $u$ such that $p(u)$ has the same sign of the leading coefficient of $p$, then the matrix $\begin{pmatrix}
p & q\\
0 & 0
\end{pmatrix}$ is a product of idempotent matrices.
\end{thm}

\begin{proof}
We start assuming that $p$ has a unique root.
Since $p$ has even degree it is not restrictive to assume that $p\geq 0$. Then we conclude by Theorem \ref{positive}.

Assume now that $q$ has a unique root $u$ and that $p(u)$ has the same sign of the leading coefficient of $p$.
We distinguish two cases.

If $\deg\, q > \deg\,p$, since $q$ has a unique root and $p$ and $q$ do not have common factors, we conclude by applying Th. \ref{teorema 3.3} (ii).

If $\deg\, p > \deg\,q$ we have two possibilities.
If $q$ does not change sign on the roots of $p$, we are done by Theorem \ref{teorema 3.3} (i). Otherwise,
by Lemma \ref{1root_pari}, it is always possible to find a suitable $r\in \R$ such that $rp+q$ has exactly two roots $z_1,z_2$ such that $x(z_1)x(z_2)>0$. Therefore, by Theorem \ref{teorema 3.3} (ii), the matrix $\begin{pmatrix}
p & rp+q\\
0 & 0
\end{pmatrix}$, similar to $\begin{pmatrix}
p & q\\
0 & 0
\end{pmatrix}$, is a product of idempotent matrices.
\end{proof}

\begin{rem}
Let $p=x/\g$ and $q=y/\g$ be elements of $D$ such that $\max\{\deg p, \deg q\}=0$. If $p$ and $q$ have a common factor $M\notin \G$, whenever the degree of $M$ is odd and $\deg\, \delta\geq 1+ \deg\, M$, the decomposition 
\[\begin{pmatrix}
p & q\\
0 & 0
\end{pmatrix}=\begin{pmatrix}
M/\delta & 0\\
0 & 0
\end{pmatrix}\begin{pmatrix}
x\delta/M\g & y\delta/M\g\\
0 & 0
\end{pmatrix}\]
is not a factorization in $D$ since $\max\{\deg (x\delta/M\g),\deg( y\delta/M\g)\}\geq 1$. For this reason, we cannot generalize Theorem \ref{even_1} to the non-comaximal case as we have done in Theorem \ref{odd_1}. 
\end{rem}

However, under the additional hypothesis that $\max\{\deg p, \deg q\}<0$, the following corollary holds.

\begin{cor}\label{corollary_even1}
Let $p=(X-z)^k\bar{x}/\g$ and $q=(X-z)^h\bar{y}/\g$, with $k,h\in\mathbb{N}^+$, $\bar{x},\bar{y}\in\R[X]$, $\g\in \G^+$, $\bar{x}(z)\ne 0$, $\bar{y}(z)\ne 0$ be two elements of $D$ such that $\max\{\deg p, \deg q\}<0$. If $\deg\,p$ is even, $\deg\,q$ is odd and either $p$ has $z$ as unique root and $\sgn (\bar{y}(z))=\sgn (\lc(\bar{y}))$ or $q$ has $z$ as unique root and $\sgn (\bar{x}(z))=\sgn (\lc(\bar{x}))$, then the matrix $\begin{pmatrix}
p & q\\
0 & 0
\end{pmatrix}$ is a product of idempotent matrices.
\end{cor}
\begin{proof} We skip the details of the proof since it is analogous to the second part of the proof of Theorem \ref{odd_1}. We reach our conclusion by properly applying Theorems \ref{positive} and \ref{even_1} and using (\ref{r0}).
\end{proof}

\begin{rem}
It is worth noting that the couples $(p,q)\in D^2$ such that $\begin{pmatrix}
p & q\\
0 & 0
\end{pmatrix}$ is a product of idempotent matrices can generate both principal and non-principal ideals of $D$. Thus, this characterization of the elements $p$ and $q$ is not related to the idempotent factorization of $\begin{pmatrix}
p & q\\
0 & 0
\end{pmatrix}$. The same fact can be observed in \cite{CZ_integers} for the factorization into idempotent factors of matrices of the form $\begin{pmatrix}
p & q\\
0 & 0
\end{pmatrix}$ over real quadratic integer rings.
\end{rem}

Theorems \ref{positive}, \ref{odd_1}, \ref{even_1} and Corollary \ref{corollary_even1} contribute to narrow down the class of singular dimension $2$ matrices over $D$ that might not admit an idempotent factorization. We provide an explicit example here below.
\begin{ex}\label{Example}
The simplest example of $2\times 2$ singular matrix over $D$ to which the above results do not apply and for which we cannot prove or disprove the existence of an idempotent factorization is the matrix $\mathbf{M}=\begin{pmatrix}
(X^2-1)/(1+X^2) & X/(1+X^2)\\
0 & 0
\end{pmatrix}$.

Nevertheless, it can bee seen without too much effort that $\mathbf{M}$ does not admit ``easy'' idempotent decompositions. 

First of all, $\mathbf{M}$ cannot factor in $M_2(D)$ as $\mathbf{M}=\begin{pmatrix}
p' & 0\\
0 & 0
\end{pmatrix}\mathbf{T}$, with $\mathbf{T}$ idempotent. 

Let $p'=x'/\eta,\,a=a'/\delta,\,b=b'/\delta,\,c=c'/\delta\in D$, $a(1-a)=bc$ and assume by contradiction that \begin{equation}\label{matrixequation}
\mathbf{M}=\begin{pmatrix}
p' & 0\\
0 & 0
\end{pmatrix}\begin{pmatrix}
a & b\\
c & 1-a
\end{pmatrix}.
\end{equation}
The matrix equation (\ref{matrixequation}) leads to the equalities
\[
(X^2-1)/X=a'/b'=c'/(\delta-a').
\]
It follows that there exist $t,s\in \R[X]$ such that $a'=(X^2-1)t$, $b'=Xt$, $c'=(X^2-1)s$ and $\delta-a'=Xs$. Therefore, 
\begin{equation}\label{1}
(X^2-1)t+Xs=\delta.
\end{equation} The assumption that $a,b,c\in D$ also implies that $\deg \,t + 2, \deg\,s +2 \leq \deg\, \delta$ hence, by (\ref{1}), $\deg\, \delta=\deg \,t + 2 < \deg\,s +1$ and $\deg\,t$ is even. Moreover, being $\delta$ a monic polynomial in $\Gamma^+$, $t$ is monic as well and, since $\delta(0)=-t(0)>0$, there exist $t_1,t_2\in\R$ roots of $t$ such that $t$ is negative in $(t_1,t_2)$. Now, from the product in (\ref{matrixequation}), we have that $(X^2-1)/(1+X^2)=p'a$, i.e., $x't=\eta\delta/(1+X^2)\in\Gamma^+$. But this is impossible since $t\notin \Gamma$.

Analogous arguments show that $\mathbf{M}$ cannot factor in $M_2(D)$ as $\mathbf{M}=\begin{pmatrix}
0 & q'\\
0 & 0
\end{pmatrix}\mathbf{T}$, with $\mathbf{T}$ idempotent.

Moreover, it is also easy to show that $\mathbf{M}$ cannot be written as a product of two idempotent matrices. If we assume by contradiction that this happens, by Lemma \ref{construction},
\begin{equation}\label{matrixequation2}
\mathbf{M}=\begin{pmatrix}
1 & q'\\
0 & 0
\end{pmatrix}\begin{pmatrix}
a & b\\
c & 1-a
\end{pmatrix}
\end{equation}
with $q',a,b,c\in D$ and $a(1-a)=bc$. Set $q'=y'/\eta,\,a=a'/\delta,\,b=b'/\delta,\,c=c'/\delta$. Arguing as in the previous case, the matrix equation (\ref{matrixequation2}) implies that there exist $t,s\in\R[X]$ such that \[(X^2-1)t+Xs=\delta\] and 
\[\eta t+ y's=\eta\delta/(1+X^2).\]
Evaluating the first equality in $+1$ and $-1$, we get $s(-1)s(1)<0$. It follows that there exists a root of $s$ $s_1\in(-1,1)$ and that $t(s_1)=\delta(s_1)/(s_1^2-1)<0$. Evaluating the second equality in $s_1$ we obtain the contradiction $t(s_1)=\delta(s_1)/(s_1^2+1)>0$.
\end{ex}

\begin{rem}\label{comments}
As recalled in the introduction, Salce and Zanardo conjectured in \cite{SalZan} that every integral domain $R$ satisfying the property \ID2 should be a B\'ezout domain. The conjecture, suggested by previous results by Laffey \cite{Laff1}, Ruitenburg \cite{Ruit} and Bhaskara Rao \cite{BhasRao}, is motivated by many positive cases. Unique factorization domains, projective-free domains,  local domains and PRINC domains (introduced in \cite{SalZan}) turn to be B\'ezout whenever they satisfy property \ID2. In \cite{CZ_Prufer} it is proved that if every singular $2\times 2$ matrix over $R$ is a product of idempotent matrices, then $R$ is a Pr\"ufer domain such that every invertible $2\times 2$ matrix over $R$ is a product of elementary matrices. Also, interesting examples of Pr\"ufer non B\'ezout domains not satisfying \ID2 were provided. 
On the other hand, the recent paper \cite{CZ_integers} raised some doubts on the general validity of the conjecture. In fact, the authors showed that the large family of dimension $2$ column-row matrices over a real quadratic integer ring $\mathcal{O}$ factorize as products of idempotent matrices, even when $\mathcal{O}$ is not a B\'ezout domain. The failure of property \ID2 for the minimal Dress ring $D$ of $\R(X)$ should be proved using matrices similar to that in the above example. 

\end{rem}

\bibliographystyle{plain}

\begin{thebibliography}{10}
%
\bibitem{BhasRao}
K.~P.~S. Bhaskara~Rao.
\newblock Products of idempotent matrices over integral domains.
\newblock {\em Linear Algebra Appl.}, 430(10): 2690--2695, 2009.
%
\bibitem{Cohn}
P.~M.~Cohn.
\newblock On the structure of the {${\rm GL}_{2}$} of a ring.
\newblock {\em Inst. Hautes \'Etudes Sci. Publ. Math.}, (30): 5--53, 1966.
%
\bibitem{CZ_Prufer}
L.~Cossu and P.~Zanardo.
\newblock Factorizations into idempotent factors of matrices over Pr\"ufer
  domains.
\newblock {\em Comm. in Algebra}, 47(4):1818-1828, 2019.

\bibitem{CZ_Dress}
L.~Cossu and P.~Zanardo.
\newblock Minimal Pr\"ufer-Dress rings and products of idempotent matrices.
\newblock {\em Houston J. Math}, 45(4): 979--994, 2019.

\bibitem{CZ_integers}
L.~Cossu and P.~Zanardo.
\newblock Idempotent factorizations of singular $2\times 2$ matrices over quadratic integer rings.
\newblock {\em Linear Multilinear Algebra}, to appear, 2020.
\newblock Available at {\href{https://www.tandfonline.com/doi/full/10.1080/03081087.2020.1721416}{https://www.tandfonline.com/doi/full/10.1080/03081087.2020.1721416}}
%
\bibitem{CZZ}
L.~Cossu, P.~Zanardo, and U.~Zannier.
\newblock Products of elementary matrices and non-{E}uclidean principal ideal
  domains.
\newblock {\em J. Algebra}, 501:182--205, 2018.
%
\bibitem{D}
A.~Dress.
\newblock Lotschnittebenen mit halbierbarem rechtem {W}inkel.
\newblock {\em Arch. Math. (Basel)}, 16:388--392, 1965.
%
%
%
%
%
%
%
\bibitem{survey}
S. K.~Jain and A.~Leroy.
\newblock Decomposition of singular elements of an algebra into product of idempotents, a survey.
\newblock Contributions in algebra and algebraic geometry, {\em Contemp. Math.}, 738: 57--74, Amer. Math. Soc., Providence, RI, 2019.
%
\bibitem{Laff1}
T.~J.~Laffey.
\newblock Products of idempotent matrices.
\newblock {\em Linear and Multilinear Algebra}, 14(4): 309--314, 1983.
%
\bibitem{Ruit}
W.~Ruitenburg.
\newblock Products of idempotent matrices over {H}ermite domains.
\newblock {\em Semigroup Forum}, 46(3):371--378, 1993.
%
\bibitem{SalZan}
L.~Salce and P.~Zanardo.
\newblock Products of elementary and idempotent matrices over integral domains.
\newblock {\em Linear Algebra Appl.}, 452:130--152, 2014.
%
%

\end{thebibliography}

\end{document}